\documentclass[12pt,reqno]{amsart}
\usepackage{amsmath,amssymb,amsthm,mathtools}
\IfFileExists{stmaryrd.sty}{\usepackage{stmaryrd}}{%
  \providecommand{\llbracket}{\mathopen{[\![}}%
  \providecommand{\rrbracket}{\mathclose{]\!]}}%
}
\usepackage{geometry}
\usepackage{enumitem,xcolor}
\usepackage[backend=biber,style=numeric,sorting=nyt]{biblatex}
\usepackage{hyperref}
\hypersetup{colorlinks=true,linkcolor=blue,citecolor=blue,urlcolor=blue,
 pdfauthor={Xianfeng JIANG},
 pdftitle={Rigidity of Euclidean Convex Fillings for Ordinary Ambrosio--Kirchheim Mass}}
\usepackage[capitalise,nameinlink,noabbrev]{cleveref}
\allowdisplaybreaks
\numberwithin{equation}{section}
\theoremstyle{plain}
\newtheorem{theorem}{Theorem}[section]
\newtheorem{lemma}[theorem]{Lemma}
\newtheorem{proposition}[theorem]{Proposition}
\newtheorem{corollary}[theorem]{Corollary}
\theoremstyle{remark}

\newcommand{\R}{\mathbb R}
\newcommand{\Lip}{\operatorname{Lip}}
\newcommand{\spt}{\operatorname{spt}}
\newcommand{\set}{\operatorname{set}}
\newcommand{\diam}{\operatorname{diam}}
\newcommand{\id}{\operatorname{id}}
\newcommand{\I}{\mathbf I}
\newcommand{\N}{\mathbf N}
\newcommand{\M}{\mathbf M}
\newcommand{\Leb}{\mathcal L}
\newcommand{\Haus}{\mathcal H}
\newcommand{\Cont}{\mathcal C}
\newcommand{\cur}[1]{\llbracket #1\rrbracket}
\newcommand{\byref}[2]{\mathrel{\overset{\text{\scriptsize\ref{#1}}}{#2}}}
\title[Ordinary AK mass filling rigidity]{Rigidity of Euclidean Convex Fillings for Ordinary Ambrosio--Kirchheim Mass}
\author{Xianfeng JIANG}
\address{Laboratoire mathématique d'Orsay, Université Paris-Saclay,
91405 Orsay, France}
\email{xianfeng.jiang@mail.ustc.edu.cn}
\date{}
\subjclass[2020]{49Q15, 53C23}
\keywords{Metric currents, Ambrosio--Kirchheim mass, minimal fillings,
convex bodies, Lipschitz-volume rigidity}
\begin{document}
\begin{abstract}
We prove rigidity for compact integral current fillings with respect to
ordinary Ambrosio--Kirchheim mass. Let $C\subset\R^n$ be a compact convex
body with nonempty interior, $n\ge1$, and let $T\in\I_n(X)$, where
$X=\spt T$ is compact. Suppose $\iota:\partial C\to X$ is an isometric
embedding for the Euclidean chord metric and
\[
 \partial T=\iota_\#\cur{\partial C},\qquad \M(T)=\Leb^n(C).
\]
Then $\iota$ extends to a surjective isometry $I:C\to X$ satisfying
$I_\#\cur C=T$. No map from the filling to the convex body is assumed.
\end{abstract}
\maketitle
\pagestyle{plain}

\section{Introduction}\label{sec:intro}
Gromov's filling volume theory \cite[Section~2]{Gromov} asks how the
distances on a boundary constrain the volume of a space spanning it.
Equality cases connect minimal fillings with boundary rigidity.
Burago--Ivanov \cite{BI} proved minimal filling and boundary rigidity
results for Riemannian metrics sufficiently close to a Euclidean metric.

The theory of metric currents of Ambrosio--Kirchheim \cite{AK} and
the integral current spaces of Sormani--Wenger \cite{SW} provide a
setting for fillings beyond smooth manifolds. In this setting, the
choice of volume is essential. For integer-rectifiable currents,
ordinary AK mass $\M$ and inscribed Riemannian mass $\M^{\mathrm{ir}}$
satisfy
\begin{equation}\label{eq:mass-distinction}
 \M(S)\le\M^{\mathrm{ir}}(S).
\end{equation}
They agree on Riemannian tangent norms but can differ on other normed
tangents \cite[Sections~2.3--2.5]{BCS}. Hence equality for $\M$ does
not imply equality for $\M^{\mathrm{ir}}$. Basso \cite{Basso} relates modified current mass axioms to Finsler
volume functionals.

Basso--Creutz--Soultanis \cite[Theorem~1.1]{BCS} proved filling
equality rigidity of Euclidean convex bodies for
$\M^{\mathrm{ir}}$. Their Lemma~6.1 gives the ordinary AK bound
$\M(T)\ge\Leb^n(C)$, and their Introduction asks whether its equality
case is rigid. Their ordinary-mass theorem
\cite[Theorem~1.2]{BCS} answers this under a prescribed
$1$-Lipschitz map to Euclidean space and boundary-current and mass
conditions. Del Nin--Perales \cite[Theorem~1.1]{DNP} prove a related
result for Euclidean balls with an injective boundary map.
Z\"ust \cite[Theorem~1.2]{Zust} treats infinitesimally Euclidean
Lipschitz targets that are essential length spaces, again with a
prescribed map and current and mass conditions.

For closed metric surfaces, Meier--Ntalampekos
\cite[Theorem~1.1]{MN} prove that surjective $1$-Lipschitz maps to
closed Riemannian surfaces with the same Hausdorff area are
isometries. Basso--Marti--Wenger \cite[Theorem~1.4]{BMW} obtain
Lipschitz-volume rigidity for closed orientable metric manifolds under
linear local contractibility in dimensions greater than two. Marti's revised
theorem \cite[Theorem~1.2]{Marti} replaces linear local contractibility by
negligibility of porous sets for nonzero-degree maps in dimensions
at least three. These results start from a prescribed $1$-Lipschitz map
from the source space to the target; their source spaces and
Hausdorff-volume conditions also differ from the current
fillings considered here.

We prove ordinary AK mass equality rigidity for fillings with compact
support. From the boundary isometry and mass equality alone, we
construct the $1$-Lipschitz Euclidean map required by
\cite[Theorem~1.2]{BCS}. No manifold structure or infinitesimal
Euclidean assumption on the filling is imposed. The resulting
isometry extends the prescribed boundary embedding and preserves
the current.

Let $C\subset\R^n$ be a compact convex body with nonempty interior,
$n\ge1$. Write $L_C=\diam C$, and let $|\cdot|$ denote the Euclidean norm.
We equip $\partial C$ with the Euclidean chord distance $|z-z'|$.
The notation $\I_n(X)$ denotes integral $n$-currents in the sense of
Ambrosio--Kirchheim; $\cur C$ is integration over $C$ with the
standard orientation and $\cur{\partial C}=\partial\cur C$.
The definitions needed in the proof are recalled in \cref{sec:prelim}.

\begin{theorem}\label{thm:main}
Let $T\in\I_n(X)$, where $X=\spt T$ is compact. Suppose that
$\iota:(\partial C,|\cdot-\cdot|)\hookrightarrow(X,d)$ is an
isometric embedding and
\begin{equation}\label{eq:main-input}
 \partial T=\iota_\#\cur{\partial C},\qquad \M(T)=\Leb^n(C).
\end{equation}
Then there is a surjective isometry $I:C\to X$ such that
\begin{equation}\label{eq:main-output}
 I|_{\partial C}=\iota,\qquad I_\#\cur C=T.
\end{equation}
\end{theorem}

In \cref{sec:chords,sec:scalar,sec:proof}, the data $C,X,T,\iota$
satisfy \cref{thm:main}, and $\mu=\|T\|$.

The proof combines scalar slicing \cite{AK,Lang}, the
one-dimensional decomposition of \cite{BDNP}, and McShane--Whitney
extensions \cite{McShane}. Mass equality identifies slices with
minimizing curves and forces uniqueness of scalar extensions.
A quadratic angular estimate yields exact rotational compatibility
by iteration. Its linearity corollary makes the coordinate map
$1$-Lipschitz, so \cite[Theorem~1.2]{BCS} applies.

\Cref{sec:prelim} recalls the current theory used in the proof.
\Cref{sec:chords,sec:scalar} construct the nonexpanding map.
\Cref{sec:proof} applies the rigidity theorem for a prescribed map and
treats $n=1$ separately.

\section{Preliminaries on metric currents and slicing}\label{sec:prelim}
We write $\Lip(f)$ for the Lipschitz constant of $f$ and consider
currents on complete metric spaces. Lang's finite-mass currents and
mass measures agree with the AK definitions on locally compact
supports \cite[(4.14)--(4.15) and the following paragraph]{Lang}.

Let $S$ be an $m$-dimensional metric current on a complete metric
space $Z$. It acts on tuples $(a,q_1,\ldots,q_m)$.
Here $a$ is bounded Lipschitz and each $q_j$ is Lipschitz.
The expression $S(a,q_1,\ldots,q_m)$ is
the metric counterpart of evaluating a classical current on
$a\,dq_1\wedge\cdots\wedge dq_m$. The action is multilinear and local.
Locality means that $S(a,q_1,\ldots,q_m)=0$ if some $q_j$ is
constant near $\{a\ne0\}$. The action is continuous when the $q_j$
converge pointwise with uniformly bounded Lipschitz constants.
The finite-mass axiom is
\begin{equation}\label{eq:mass-axiom}
 |S(a,q_1,\ldots,q_m)|
 \le\prod_{j=1}^m\Lip(q_j)\int |a|\,d\|S\|.
\end{equation}
Here $\|S\|$ is the least Borel measure for which this estimate holds,
and $\M(S)=\|S\|(Z)$ is the ordinary AK mass. By
\eqref{eq:mass-axiom}, the action of $S$ in its first argument $a$
extends continuously to $L^1(\|S\|)$. In particular, bounded Borel
functions can be used in this argument \cite[Section~3, Theorem~3.5]{AK}.
The support is the closed set
\[
 \spt S=\spt\|S\|
 =\{x:\|S\|(B_r(x))>0\text{ for every }r>0\}.
\]
We write $\N_m(Z)$ for normal currents: these have finite mass and
finite boundary mass. Integral currents, denoted by $\I_m(Z)$,
are the integer-rectifiable normal currents. An integer-rectifiable
current is a countable sum of pushforwards of Euclidean currents under
bi-Lipschitz maps from compact subsets of $\R^m$ into $Z$
\cite[Theorem~4.5]{AK}. Each piece has an integrable integer-valued
multiplicity on its parameter domain. The chart images cover
$\|S\|$-almost all of $Z$.
Normality alone does not imply integrality.

For $m\ge1$, the positive-density set is
\begin{equation}\label{eq:canonical-set}
 \set S=\left\{x:
 \liminf_{r\downarrow0}\frac{\|S\|(B_r(x))}{r^m}>0\right\}.
\end{equation}
For an integral current, $\|S\|$ is concentrated on $\set S$
\cite[Theorem~4.6]{AK}; this set need not be closed.
An integral current space is the pair $(Y,S)$ with $Y=\set S$ and
$S$ an integral current on the completion of $Y$; compare
\cite[Definition~2.46]{SW}.

For $m\ge1$, the boundary is the $(m-1)$-current
\[
 \partial S(a,q_1,\ldots,q_{m-1})=S(1,a,q_1,\ldots,q_{m-1}).
\]
For a Lipschitz map $f:Z\to W$, pushforward preserves dimension and is
defined by
\[
 f_\#S(a,q_1,\ldots,q_m)
 =S(a\circ f,q_1\circ f,\ldots,q_m\circ f).
\]
These definitions give
\begin{equation}\label{eq:push-boundary}
 \partial(f_\#S)=f_\#(\partial S),\qquad
 \|f_\#S\|\le\Lip(f)^m f_\#\|S\|.
\end{equation}
If $f$ is an isometric embedding, then
$\|f_\#S\|=f_\#\|S\|$. Apply \eqref{eq:push-boundary} to $f$
and to its inverse on $f(Z)$. For the inverse estimate, every
real-valued Lipschitz test on $f(Z)$ extends to $W$ with the same
Lipschitz constant \cite{McShane}.

For a Borel set $A$, restriction $S\llcorner A$ means multiplying
the first argument by $\mathbf1_A$; it has dimension $m$ and need not
be normal. More generally, contraction by $k\le m$ scalar functions means
\begin{equation}\label{eq:contraction}
 \begin{aligned}
 R&=S\llcorner(1,p_1,\ldots,p_k),\\
 R(a,q_1,\ldots,q_{m-k})
   &=S(a,p_1,\ldots,p_k,q_1,\ldots,q_{m-k}).
 \end{aligned}
\end{equation}
It has dimension $m-k$. These operations and the extension in the
first argument are as in \cite[Sections~2--3]{AK}.

A zero-dimensional integral current of finite mass is a finite sum
\[
 Q=\sum_{j=1}^N m_j\cur{x_j},\qquad
 Q(a)=\sum_{j=1}^N m_j a(x_j),\qquad
 \M(Q)=\sum_{j=1}^N|m_j|,
\]
with distinct $x_j$ and nonzero integers $m_j$. In particular, two
different unit atoms cannot cancel. For a Lipschitz curve
$\gamma:[0,L]\to X$, its current is
\begin{equation}\label{eq:curve-current}
 \cur\gamma(a,q)=\int_0^L a(\gamma(t))(q\circ\gamma)'(t)\,dt,
 \qquad \partial\cur\gamma=\cur{\gamma(L)}-\cur{\gamma(0)}.
\end{equation}
Always $\M(\cur\gamma)\le\operatorname{length}(\gamma)$.
For an injective arclength-parametrized curve, equality holds and
$\|\cur\gamma\|=\gamma_\#(\Leb^1\llcorner[0,L])$.
These are the one-dimensional area and boundary formulas; see
\cite[Section~2.3 and Lemma~5.4]{BDNP}. A minimizing curve is a curve whose length
equals the distance between its endpoints. Its arclength
parametrization is isometric, as we verify in \cref{lem:two-endpoints}.

The Euclidean current of the oriented convex body $C$ is
\begin{equation}\label{eq:euclidean-density}
 \cur C(a,q_1,\ldots,q_n)
 =\int_C a(x)\det D(q_1,\ldots,q_n)(x)\,dx.
\end{equation}
A convex body with nonempty interior is a compact Lipschitz manifold
with boundary. Thus
\begin{equation}\label{eq:convex-current}
 \cur C\in\I_n(\R^n),\qquad
 \partial\cur C=\cur{\partial C},\qquad
 \M(\cur C)=\Leb^n(C),\quad\M(\cur{\partial C})=\Haus^{n-1}(\partial C);
\end{equation}
see \cite[Example~2.5 and the paragraph following it]{BCS}.
Hausdorff measure is normalized
to agree with Lebesgue measure on Euclidean planes, and $\Haus^0$
counts points. Boundary orientation is the one determined by
$\partial\cur C$.

For a normal $m$-current $S$ and a Lipschitz map
$P=(p_1,\ldots,p_k):Z\to\R^k$, $1\le k\le m$,
the slice $\langle S,P,y\rangle$ has dimension $m-k$ and exists
for almost every $y$. For one scalar coordinate our convention is
\[
 \langle S,p,t\rangle
 :=\partial(S\llcorner\{p<t\})-(\partial S)\llcorner\{p<t\}
 \quad\text{for a.e. }t.
\]
For $P=(p_1,\ldots,p_k)$ and $y=(y_1,\ldots,y_k)$, we iterate
in this coordinate order:
\[
 \langle S,P,y\rangle
 :=\bigl\langle\cdots
   \bigl\langle\langle S,p_1,y_1\rangle,p_2,y_2\bigr\rangle
   \cdots,p_k,y_k\bigr\rangle.
\]
By \cite[(6.1), (6.6) and Theorem~6.5]{Lang}, these currents
coincide with the slices of \cite[Definition~6.3]{Lang}, as
$(m-k)$-dimensional currents, for $\Leb^k$-almost every $y\in\R^k$.
The properties we use are
\begin{align}
 \spt\langle S,P,y\rangle
   &\subset \spt S\cap P^{-1}(y),\label{eq:slice-support}\\
 \partial\langle S,P,y\rangle
   &=(-1)^k\langle\partial S,P,y\rangle\quad(k<m),
       \label{eq:slice-boundary}\\
 f_\#\langle S,P\circ f,y\rangle
   &=\langle f_\#S,P,y\rangle,
       \label{eq:slice-natural}\\
 \int_{\R^k}\|\langle S,P,y\rangle\|(A)\,dy
   &=\|S\llcorner(1,p_1,\ldots,p_k)\|(A).
       \label{eq:slice-exact}
\end{align}
In \eqref{eq:slice-natural}, $f:Z\to W$ and $P:W\to\R^k$. For the support and boundary identities \eqref{eq:slice-support} and \eqref{eq:slice-boundary}, see \cite[(6.8)--(6.9)]{Lang}. Naturality \eqref{eq:slice-natural} is stated in \cite[(2.12)]{BCS}. The mass-measure equality \eqref{eq:slice-exact} is \cite[Theorem~6.4(3)]{Lang}, for every Borel set $A$, on locally compact spaces with separable $\spt S$. Local normality follows from \cite[Theorem~6.4(1)]{Lang}. For normal $S$ and Lipschitz $P$, the slices exist and are normal almost everywhere if $S$ and $\partial S$ are concentrated on a common $\sigma$-compact set \cite[Theorem~5.6]{AK}, as is the case here. For integral $S$, \cite[Theorem~5.7]{AK} also gives integer-rectifiable slices, so the slices are integral almost everywhere.

We also use the coefficient form of the slicing identity: for a bounded Borel function $h$ on $\R^k$,
\begin{equation}\label{eq:slice-disintegration}
 \int h(y)\langle S,P,y\rangle(a,q)\,dy
   =S(h(P)a,p_1,\ldots,p_k,q).
\end{equation}
Here $q$ stands for the remaining $m-k$ scalar tests. For almost every $y$, \eqref{eq:slice-support} gives
\[
 \langle S,P,y\rangle(h(P)a,q)=h(y)\langle S,P,y\rangle(a,q),
\]
and \cite[Theorem~6.4(2)]{Lang}, applied to $h(P)a$, gives \eqref{eq:slice-disintegration}. The local versions suffice for Euclidean currents with compact support.

For integer-rectifiable $S$, a family of finite-mass $(m-k)$-currents is characterized almost everywhere by \eqref{eq:slice-disintegration} for all test tuples if its test
evaluations are measurable, its masses are integrable, and it is
concentrated on the corresponding fibers inside a common
$\sigma$-compact subset of $Z$. It suffices to use $h\in C_c(\R^k)$
\cite[Section~2.6, after (2.11)]{BCS}.

We use the following rigidity theorem of Basso--Creutz--Soultanis
\cite[Theorem~1.2, p.~214]{BCS}, stated with ordinary AK mass.
\begin{theorem}\label{thm:given-map}
Let $(Y,S)$ be an $n$-dimensional integral current space, let
$C\subset\R^n$ be a convex body, and let $f:Y\to\R^n$ be
$1$-Lipschitz. Suppose
\[
 f_\#\partial S=\cur{\partial C},\qquad
 \M(S)\le \Leb^n(C),\qquad \M(\partial S)\le \Haus^{n-1}(\partial C).
\]
Then $f$ is an isometry onto $C$.
\end{theorem}
When $Y$ is not complete, pushforward is defined using the unique Lipschitz
extension to its completion, as specified in \cite[Section~2.4]{BCS}.

\section{Sharp slicing and minimizing curves}\label{sec:chords}
Equations \eqref{eq:mass-axiom} and \eqref{eq:slice-exact} give the
sharp local slicing bound in \cref{lem:sharp}. The boundary
formula identifies the endpoints of the slices. Mass equality then
represents almost every slice by a minimizing curve in $X$.

\begin{lemma}\label{lem:sharp}
Let $Z$ be a compact metric space and $S\in\N_m(Z)$. Let $P=(p_1,\ldots,p_k):Z\to\R^k$ have
$\Lip(p_i)\le1$, where $1\le k\le m$. For every Borel $A\subset Z$,
\begin{equation}\label{eq:sharp}
 \lambda_P(A):=\int_{\R^k}\|\langle S,P,y\rangle\|(A)\,dy
 =\|S\llcorner(1,p_1,\ldots,p_k)\|(A)\le\|S\|(A).
\end{equation}
\end{lemma}
\begin{proof}
Set $R=S\llcorner(1,p_1,\ldots,p_k)$. By \eqref{eq:mass-axiom},
for bounded Borel $a$ and Lipschitz scalar tests $q_1,\ldots,q_{m-k}$,
\[
 |R(a,q_1,\ldots,q_{m-k})|
 \le\prod_{j=1}^{m-k}\Lip(q_j)\int_Z|a|\,d\|S\|.
\]
Thus $\|R\|\le\|S\|$ by the defining minimality of the mass measure.
The identity \eqref{eq:slice-exact} gives $\lambda_P=\|R\|$.
\end{proof}

For the transverse constructions, fix $n\ge2$, a unit vector $v$,
and an ordered orthonormal basis $b_1,\ldots,b_{n-1}$ of $v^\perp$.
We identify $v^\perp$ with $\R^{n-1}$ in this basis and write
\begin{equation}\label{eq:transverse-convention}
 \pi_v(z)=(b_1\cdot z,\ldots,b_{n-1}\cdot z),\qquad
 \sigma=\det(b_1,\ldots,b_{n-1},v)\in\{-1,1\}.
\end{equation}
Set $K_v=\pi_v(C)$. For $y\in K_v$, compactness and convexity of $C$
make $C\cap\pi_v^{-1}(y)$ a closed line segment, possibly a single
point. Denote its endpoints by $z_y^-,z_y^+$, ordered so that
$v\cdot z_y^-\le v\cdot z_y^+$.
For $y\in\operatorname{int}K_v$, this is a nondegenerate maximal
chord of $C$ parallel to $v$; set $\ell_y:=|z_y^+-z_y^-|$. Then
\begin{equation}\label{eq:chord-notation}
 C\cap\pi_v^{-1}(y)=[z_y^-,z_y^+],\qquad
 z_y^+-z_y^-=\ell_yv,\quad 0<\ell_y\le L_C.
\end{equation}
A surjective linear projection maps $\operatorname{int}C$ onto
$\operatorname{int}K_v$, so these chords are nondegenerate.
Both endpoints lie on $\partial C$.

For every $s\in\R$,
\[
 \begin{aligned}
 \{y\in K_v:v\cdot z_y^+\ge s\}
   &=\pi_v(C\cap\{z:v\cdot z\ge s\}),\\
 \{y\in K_v:v\cdot z_y^-\le s\}
   &=\pi_v(C\cap\{z:v\cdot z\le s\}).
 \end{aligned}
\]
The right-hand sets are compact images of closed subsets of $C$
under the continuous map $\pi_v$, so $y\mapsto v\cdot z_y^\pm$ are
Borel. The identity
$z_y^\pm=\sum_{j=1}^{n-1}y_jb_j+(v\cdot z_y^\pm)v$
then makes both endpoint maps Borel.

Set $\ell_y=0$ outside $\operatorname{int}K_v$. In the orthogonal
coordinates fixed above, Fubini gives
\begin{equation}\label{eq:cavalieri}
 \begin{aligned}
 \Leb^n(C)
 &=\int_{\R^{n-1}}\int_{\R}
 \mathbf1_C\left(\sum_{j=1}^{n-1}y_jb_j+tv\right)\,dt\,dy\\
 &=\int_{\R^{n-1}}\ell_y\,dy.
 \end{aligned}
\end{equation}
The coordinate change has absolute Jacobian $1$, and the inner
integral equals the chord length off $\partial K_v$.
Since $\partial K_v$ has Lebesgue measure zero, the zero extension
does not change the integral.

Let $P=(p_1,\ldots,p_{n-1}):X\to\R^{n-1}$ satisfy
\begin{equation}\label{eq:transverse-map}
 P\circ\iota=\pi_v|_{\partial C},\qquad
 \Lip(p_j)\le1\quad(j=1,\ldots,n-1).
\end{equation}
Set $S_y=\langle T,P,y\rangle$.

\begin{lemma}\label{lem:endpoints}
For almost every $y\in\operatorname{int}K_v$,
\begin{equation}\label{eq:endpoint-boundary}
 \partial S_y
 =\sigma\bigl(\cur{\iota(z_y^+)}-\cur{\iota(z_y^-)}\bigr),
 \qquad \M(\partial S_y)=2.
\end{equation}
\end{lemma}
\begin{proof}
The coordinates $z=\sum_jy_jb_j+sv$ have orientation $\sigma$.
For bounded Lipschitz $a$, Lipschitz $q$, and bounded Borel $h$,
Fubini and \eqref{eq:euclidean-density} give
\[
 \cur C(h(\pi_v)a,\pi_v^1,\ldots,\pi_v^{n-1},q)
 =\sigma\int_{\operatorname{int}K_v}h(y)\int_0^{\ell_y}
 a(z_y^-+tv)\frac{d}{dt}q(z_y^-+tv)\,dt\,dy.
\]
The resulting segment currents are supported on the fibers of $\pi_v$.
Their masses are integrable by \eqref{eq:cavalieri}.
Extend this family by zero outside $\operatorname{int}K_v$.
By the uniqueness criterion following \eqref{eq:slice-disintegration},
the preceding Fubini identity identifies this family with the slices:
\begin{equation}\label{eq:euclidean-chord-slice}
 \langle\cur C,\pi_v,y\rangle
 =\sigma\cur{[z_y^-,z_y^+]}
 \quad\text{for a.e. }y\in\operatorname{int}K_v.
\end{equation}
The segment is directed from $z_y^-$ to $z_y^+$.
For almost every such $y$,
\begin{align}
 \partial S_y
 &\overset{\eqref{eq:slice-boundary}}{=}
 (-1)^{n-1}\langle\partial T,P,y\rangle\nonumber\\
 &=(-1)^{n-1}\langle\iota_\#\cur{\partial C},P,y\rangle\nonumber\\
 &\overset{\eqref{eq:slice-natural}}{=}
 (-1)^{n-1}\iota_\#\langle\cur{\partial C},\pi_v,y\rangle\nonumber\\
 &\overset{\eqref{eq:slice-boundary}}{=}
 (-1)^{2(n-1)}\iota_\#\partial\langle\cur C,\pi_v,y\rangle\nonumber\\
 &\overset{\eqref{eq:euclidean-chord-slice}}{=}
 \sigma\bigl(\cur{\iota(z_y^+)}-\cur{\iota(z_y^-)}\bigr).
 \label{eq:endpoint-chain}
\end{align}
Here $P\circ\iota=\pi_v$ and $\cur{\partial C}=\partial\cur C$.
The endpoints are distinct, and $\iota$ is injective.
The zero-dimensional mass formula gives $\M(\partial S_y)=2$.
\end{proof}

\begin{lemma}\label{lem:two-endpoints}
Let $(Z,d)$ be a complete metric space and $S\in\I_1(Z)$. If $a\ne b$,
$\partial S=\cur b-\cur a$, and $\M(S)=d(a,b)$, then
$S=\cur\gamma$ for an isometric curve
$\gamma:[0,d(a,b)]\to Z$ from $a$ to $b$.
\end{lemma}
\begin{proof}
The optimal representation theorem \cite[Theorem~5.3]{BDNP} gives
$S=\sum_i\cur{\gamma_i}$, where
the $\gamma_i$ are injective Lipschitz curves or injective Lipschitz
loops, and
\begin{equation}\label{eq:normal-additivity}
 \M(S)+\M(\partial S)
 =\sum_i\bigl(\M(\cur{\gamma_i})+
                     \M(\partial\cur{\gamma_i})\bigr).
\end{equation}
Mass subadditivity for the current and its boundary gives separately
\[
 \M(S)\le\sum_i\M(\cur{\gamma_i}),\qquad
 \M(\partial S)\le\sum_i\M(\partial\cur{\gamma_i}).
\]
Combining \eqref{eq:normal-additivity} with these two
mass-subadditivity inequalities gives
\begin{equation}\label{eq:separate-additivity}
 \M(S)=\sum_i\M(\cur{\gamma_i}),\qquad
 \M(\partial S)=\sum_i\M(\partial\cur{\gamma_i}).
\end{equation}
Each nonclosed injective curve has boundary mass two, and each loop has
boundary zero. The second equality gives $2=2N$, where $N$ is the
number of nonclosed curves. Denote the unique nonclosed curve by
$\gamma_{i_0}$. Since all other terms have zero boundary,
\[
 \partial\cur{\gamma_{i_0}}=\partial S=\cur b-\cur a.
\]
Thus $\gamma_{i_0}$ starts at $a$ and ends at $b$.
Let $\gamma:[0,L]\to Z$ be its arclength reparametrization, where
$L=\operatorname{length}(\gamma_{i_0})$.
Reparametrization preserves the induced current:
$\cur\gamma=\cur{\gamma_{i_0}}$.
By \eqref{eq:separate-additivity},
\[
 \M(\cur\gamma)=\M(\cur{\gamma_{i_0}})
 \le\sum_i\M(\cur{\gamma_i})=\M(S).
\]
The curve mass formula now gives
\[
 d(a,b)\le L=\M(\cur\gamma)
 \le\M(S)=d(a,b).
\]
Hence $L=d(a,b)$, and the sum of all remaining masses is zero.
All remaining currents vanish, so $S=\cur\gamma$.
For $0\le s<t\le L$, the triangle inequality gives
\[
 L=d(a,b)
 \le d(a,\gamma(s))+d(\gamma(s),\gamma(t))+d(\gamma(t),b)
 \le s+d(\gamma(s),\gamma(t))+L-t.
\]
Thus $d(\gamma(s),\gamma(t))\ge t-s$. The arclength parametrization gives
\[
 d(\gamma(s),\gamma(t))
 \le\operatorname{length}(\gamma|_{[s,t]})=t-s.
\]
Hence $d(\gamma(s),\gamma(t))=t-s$, so $\gamma$ is isometric.
\end{proof}

\begin{lemma}\label{lem:chords}
For almost every $y\in\operatorname{int}K_v$, there is an isometric
curve $\gamma_y:[0,\ell_y]\to X$ from $\iota(z_y^-)$ to
$\iota(z_y^+)$ such that $S_y=\sigma\cur{\gamma_y}$.
Moreover, for every Borel $A\subset X$,
\begin{equation}\label{eq:full-mass}
 \mu(A)=\int_{\R^{n-1}}\|S_y\|(A)\,dy.
\end{equation}
\end{lemma}
\begin{proof}
The slices are integral for almost every $y$. By \cref{lem:endpoints},
the oriented slice $\sigma S_y$ has boundary
$\cur b-\cur a$, where $a=\iota(z_y^-)$ and $b=\iota(z_y^+)$.
Set $q(x)=\min\{d(a,x),d(a,b)\}$. It is bounded, and
\[
 q(a)=0,\qquad q(b)=d(a,b),\qquad \Lip(q)\le1.
\]
The boundary formula and $|\sigma|=1$ give
\begin{equation}\label{eq:endpoint-lower}
 \begin{aligned}
 d(a,b)&=q(b)-q(a)=\partial(\sigma S_y)(q)\\
       &=(\sigma S_y)(1,q)
         \overset{\eqref{eq:mass-axiom}}{\le}\M(\sigma S_y)=\M(S_y).
 \end{aligned}
\end{equation}
Since $\iota$ is isometric and $|v|=1$,
\[
 \begin{aligned}
 d(a,b)&=d(\iota(z_y^-),\iota(z_y^+))
       \byref{thm:main}{=}|z_y^+-z_y^-|\\
       &\overset{\eqref{eq:chord-notation}}{=}\ell_y|v|=\ell_y.
 \end{aligned}
\]
By \eqref{eq:main-input} and \eqref{eq:cavalieri},
\begin{equation}\label{eq:sharp-comparison}
 \begin{aligned}
 \M(T)&\byref{eq:main-input}{=}\Leb^n(C)
       \byref{eq:cavalieri}{=}\int_{\operatorname{int}K_v}\ell_y\,dy\\
 &\byref{eq:endpoint-lower}{\le}
       \int_{\operatorname{int}K_v}\M(S_y)\,dy
       \le\int_{\R^{n-1}}\M(S_y)\,dy
       \byref{lem:sharp}{\le}\M(T).
 \end{aligned}
\end{equation}
All inequalities are equalities, so
\begin{equation}\label{eq:slice-masses}
 \M(S_y)=
 \begin{cases}
  \ell_y,&\text{for a.e. }y\in\operatorname{int}K_v,\\
  0,&\text{for a.e. }y\notin\operatorname{int}K_v.
 \end{cases}
\end{equation}
Apply \cref{lem:two-endpoints} to $\sigma S_y$ to obtain $\gamma_y$.
Since $S_y=\sigma\cur{\gamma_y}$ and $\gamma_y$ is isometric,
the curve mass formula gives
\[
 \|S_y\|=(\gamma_y)_\#(\Leb^1\llcorner[0,\ell_y]).
\]
The full support of the interval measure gives the equality, and
\eqref{eq:slice-support} gives the inclusion:
\[
 \spt S_y=\gamma_y([0,\ell_y])\subset P^{-1}(y).
\]
With $S=T$ in \eqref{eq:sharp}, $\lambda_P\le\mu$.
Equality throughout \eqref{eq:sharp-comparison} gives
$\lambda_P(X)=\mu(X)=\Leb^n(C)$.
The positive measure $\mu-\lambda_P$ has total mass zero, so
$\lambda_P=\mu$.
\end{proof}

\section{Scalar extensions and angular compatibility}\label{sec:scalar}
We write
\[
 \Cont(X)=\{f:X\to\R:f\text{ is continuous}\},\qquad
 \|f\|_\infty=\sup_{x\in X}|f(x)|.
\]
For $Z\in\Cont(X;\R^2)$ we use the Euclidean supremum norm
$\|Z\|_\infty=\sup_{x\in X}|Z(x)|$.

Let $A\subset X$ be nonempty and let $f:A\to\R$ be $1$-Lipschitz.
The McShane--Whitney formulas \cite{McShane} are
\[
 E^+(x)=\inf_{a\in A}\{f(a)+d(x,a)\},\qquad
 E^-(x)=\sup_{a\in A}\{f(a)-d(x,a)\}.
\]
Both functions are $1$-Lipschitz on $X$ and agree with $f$ on $A$.

For $v\in S^{n-1}$, apply these formulas on $A=\iota(\partial C)$
to $\iota(z)\mapsto v\cdot z$. This function is $1$-Lipschitz because
\[
 |v\cdot(z-z')|\le|z-z'|=d(\iota(z),\iota(z')).
\]
Denote the extensions by $E_v^\pm:X\to\R$ and set
\begin{equation}\label{eq:extensions}
 u_v=\tfrac12(E_v^++E_v^-),\qquad
 \Lip(u_v)\le1,\qquad u_v(\iota(z))=v\cdot z
 \quad(z\in\partial C).
\end{equation}
For $x\in X$, the infimum and supremum formulas give
\[
 \begin{aligned}
 E_{-v}^+(x)
 &=\inf_{z\in\partial C}\{-v\cdot z+d(x,\iota(z))\}\\
 &=-\sup_{z\in\partial C}\{v\cdot z-d(x,\iota(z))\}
 =-E_v^-(x).
 \end{aligned}
\]
Replacing $v$ by $-v$ gives $E_{-v}^-=-E_v^+$, so $u_{-v}=-u_v$.
Define
\begin{equation}\label{eq:initial-map}
 F=(u_{e_1},\ldots,u_{e_n}):X\to\R^n,\qquad
 F\circ\iota=\id_{\partial C}.
\end{equation}
For $x,x'\in X$,
\[
 |F(x)-F(x')|^2
 =\sum_{j=1}^n|u_{e_j}(x)-u_{e_j}(x')|^2
 \le n\,d(x,x')^2.
\]
Thus $\Lip(F)\le\sqrt n$.

For any unit $v$ and ordered orthonormal basis $b_1,\ldots,b_{n-1}$
of $v^\perp$, define
\begin{equation}\label{eq:Pv}
 P_v=(u_{b_1},\ldots,u_{b_{n-1}}).
\end{equation}
By \eqref{eq:extensions}, $P_v$ satisfies the conditions
in \eqref{eq:transverse-map}. We may therefore apply
\cref{lem:endpoints,lem:chords} with $P=P_v$ for each choice of basis.

\begin{lemma}\leavevmode\phantomsection\label{lem:unique}\ignorespaces
For every $v\in S^{n-1}$, $u_v$ is the unique $1$-Lipschitz
function $w:X\to\R$ satisfying
\[
 w(\iota(z))=v\cdot z\qquad(z\in\partial C).
\]
\end{lemma}
\begin{proof}
Let $w:X\to\R$ be $1$-Lipschitz with
$w(\iota(z))=v\cdot z$ for $z\in\partial C$.
Apply \cref{lem:chords} with $P=P_v$.
For almost every $y\in\operatorname{int}K_v$, let
$\gamma_y:[0,\ell_y]\to X$ be the isometric curve from that lemma.
Its endpoint values satisfy
\[
 w(\gamma_y(0))=v\cdot z_y^-,\qquad
 w(\gamma_y(\ell_y))=v\cdot z_y^+=v\cdot z_y^-+\ell_y.
\]
Since $\gamma_y$ is isometric,
\[
 d(\gamma_y(t),\gamma_y(0))=t,\qquad
 d(\gamma_y(t),\gamma_y(\ell_y))=\ell_y-t.
\]
The $1$-Lipschitz property of $w$ gives
\begin{align*}
 w(\gamma_y(t))&\le w(\gamma_y(0))+t=v\cdot z_y^-+t,\\
 w(\gamma_y(t))&\ge w(\gamma_y(\ell_y))-(\ell_y-t)\\
 &=v\cdot z_y^+-(\ell_y-t)=v\cdot z_y^-+t.
\end{align*}
Since $u_v$ is also $1$-Lipschitz and has the same boundary values,
the same bounds hold for $u_v$. Thus
\begin{equation}\label{eq:chord-scalar}
 w(\gamma_y(t))=u_v(\gamma_y(t))=v\cdot z_y^-+t,
 \qquad 0\le t\le\ell_y.
\end{equation}
The set $A=\{x\in X:w(x)\ne u_v(x)\}$ is open, since both functions
are continuous. By \eqref{eq:chord-scalar},
$\|\langle T,P_v,y\rangle\|(A)=0$ for almost every
$y\in\operatorname{int}K_v$.
Equation \eqref{eq:slice-masses} gives
$\|\langle T,P_v,y\rangle\|=0$ for almost every
$y\notin\operatorname{int}K_v$. Hence \eqref{eq:full-mass} gives
\[
 \mu(A)=\int_{\R^{n-1}}
             \|\langle T,P_v,y\rangle\|(A)\,dy=0.
\]
Since $X=\spt\mu$, every nonempty open subset of $X$ has positive
$\mu$-measure. Thus $A=\varnothing$ and $w=u_v$ on $X$.
\end{proof}

\begin{proposition}\label{prop:angle}
For every orthonormal pair $v,w\in\R^n$ and every $\theta\in\R$,
put $v_\theta=\cos\theta\,v+\sin\theta\,w$. Then
\begin{equation}\label{eq:angle}
 \|u_{v_\theta}-\cos\theta\,u_v-\sin\theta\,u_w\|_\infty
       \le L_C(1-\cos\theta).
\end{equation}
\end{proposition}
\begin{proof}
Choose the transverse basis with $b_1=w$ in \eqref{eq:Pv}.
For almost every $y\in\operatorname{int}K_v$, take the isometric
curve $\gamma_y$ given by \cref{lem:chords} with $P=P_v$.
Equation \eqref{eq:chord-scalar} gives
\[
 u_v(\gamma_y(t))=v\cdot z_y^-+t.
\]
The image of $\gamma_y$ lies in $P_v^{-1}(y)$.
Since the first component of $P_v$ is $u_w$,
\begin{equation}\label{eq:transverse-constant}
 u_w(\gamma_y(t))=u_w(\iota(z_y^-))=w\cdot z_y^-.
\end{equation}

The function $u_{v_\theta}$ is $1$-Lipschitz and satisfies
$u_{v_\theta}(\iota(z))=v_\theta\cdot z$ on $\partial C$.
The $1$-Lipschitz bounds for $u_{v_\theta}$ at the two endpoints of
$\gamma_y$ are
\[
 v_\theta\cdot z_y^+-(\ell_y-t)
    \le u_{v_\theta}(\gamma_y(t))
    \le v_\theta\cdot z_y^-+t.
\]
Equations \eqref{eq:chord-scalar} and \eqref{eq:transverse-constant}
give
\[
 \cos\theta\,u_v(\gamma_y(t))+\sin\theta\,u_w(\gamma_y(t))
 =v_\theta\cdot z_y^-+t\cos\theta.
\]
Using $v_\theta\cdot(z_y^+-z_y^-)=\ell_y\cos\theta$, the differences
between the endpoint bounds and this identity are
\[
 \begin{aligned}
 &v_\theta\cdot z_y^+-(\ell_y-t)
      -(v_\theta\cdot z_y^-+t\cos\theta)\\
 &\qquad=\ell_y\cos\theta-\ell_y+t-t\cos\theta
       =-(1-\cos\theta)(\ell_y-t),\\
 &(v_\theta\cdot z_y^-+t)
      -(v_\theta\cdot z_y^-+t\cos\theta)=(1-\cos\theta)t.
 \end{aligned}
\]
Therefore
\begin{equation}\label{eq:angle-interval}
 \begin{aligned}
 -(1-\cos\theta)(\ell_y-t)
 &\le u_{v_\theta}(\gamma_y(t))
       -\cos\theta\,u_v(\gamma_y(t))
       -\sin\theta\,u_w(\gamma_y(t))\\
 &\le(1-\cos\theta)t.
 \end{aligned}
\end{equation}
As $0\le t\le\ell_y\le L_C$, \eqref{eq:angle-interval} bounds
the absolute error by $L_C(1-\cos\theta)$ on these curves.
By \eqref{eq:full-mass}, the bound holds $\mu$-almost everywhere.
The error function is continuous. Since $\spt\mu=X$, the bound
holds everywhere on $X$.
\end{proof}

\begin{proposition}\label{prop:rotation}
For every orthonormal pair $v,w$ and every real $\theta$,
\begin{equation}\label{eq:exact-rotation}
 u_{\cos\theta v+\sin\theta w}
      =\cos\theta\,u_v+\sin\theta\,u_w\quad\text{on }X.
\end{equation}
\end{proposition}
\begin{proof}
Fix $v,w,\theta$ and a positive integer $N$. Put $h=\theta/N$ and
\[
 v_k=\cos(kh)v+\sin(kh)w,\qquad
 w_k=-\sin(kh)v+\cos(kh)w,\qquad Z_k=(u_{v_k},u_{w_k}).
\]
Put
\[
 R_h=\begin{pmatrix}\cos h&\sin h\\-\sin h&\cos h\end{pmatrix}.
\]
Set $E_k=Z_{k+1}-R_hZ_k$.
Apply \cref{prop:angle} to $(v_k,w_k)$ for the first component of $E_k$.
For the second component,
\[
 w_{k+1}=\cos h\,w_k+\sin h(-v_k),\qquad u_{-v_k}=-u_{v_k}.
\]
Thus
\[
 \begin{aligned}
 E_k^{(2)}
 &=u_{w_{k+1}}-\cos h\,u_{w_k}+\sin h\,u_{v_k}\\
 &=u_{\cos h\,w_k+\sin h(-v_k)}
     -\cos h\,u_{w_k}-\sin h\,u_{-v_k}.
 \end{aligned}
\]
Apply \cref{prop:angle} to $(w_k,-v_k)$.
Each component is bounded in absolute value by $L_C(1-\cos h)$, so
\begin{equation}\label{eq:iteration-error}
 Z_{k+1}=R_hZ_k+E_k,\qquad
 \|E_k\|_\infty\le\sqrt2 L_C(1-\cos h).
\end{equation}
For $z\in\R^2$, $|R_hz|=|z|$. Thus
$\|R_hZ\|_\infty=\|Z\|_\infty$ for $Z\in\Cont(X;\R^2)$.
Since $R_h^N=R_\theta$, telescoping gives
\begin{align}
 Z_N-R_\theta Z_0&=\sum_{k=0}^{N-1}R_h^{N-1-k}E_k,\nonumber\\
 \|Z_N-R_\theta Z_0\|_\infty
 &\le\sqrt2 NL_C(1-\cos(\theta/N))
 \le\frac{L_C\theta^2}{\sqrt2N}.
 \label{eq:iteration-limit}
\end{align}
The functions $Z_N=(u_{v_\theta},u_{-\sin\theta v+\cos\theta w})$
and $R_\theta Z_0$ are independent of $N$. Hence \eqref{eq:iteration-limit}
gives $\|Z_N-R_\theta Z_0\|_\infty=0$, proving \eqref{eq:exact-rotation}.
\end{proof}

\begin{corollary}\label{cor:linearity}
The homogeneous extension
\begin{equation}\label{eq:homogeneous}
 U(a)=\begin{cases}|a|u_{a/|a|},&a\ne0,\\0,&a=0,
       \end{cases}
 \qquad U:\R^n\longrightarrow \Cont(X)
\end{equation}
is a real linear map.
\end{corollary}
\begin{proof}
For an orthonormal pair $v,w$ and $(r,s)\ne(0,0)$, choose
$\rho=\sqrt{r^2+s^2}$ and $\theta$ with
$(r,s)=\rho(\cos\theta,\sin\theta)$. Then
\[
 U(rv+sw)=\rho u_{\cos\theta v+\sin\theta w}
       \overset{\eqref{eq:exact-rotation}}{=}r u_v+s u_w.
\]
This also holds when $r=s=0$, since $U(0)=0$.
For $a,b\in\R^n$, choose a two-dimensional subspace containing them;
this is possible since $n\ge2$. In an orthonormal basis $v,w$ of this
subspace, write $a=rv+sw$ and $b=r'v+s'w$. For every $t\in\R$,
\[
 \begin{aligned}
 U(ta+b)&=U((tr+r')v+(ts+s')w)\\
       &=(tr+r')u_v+(ts+s')u_w\\
       &=t(ru_v+su_w)+(r'u_v+s'u_w)=tU(a)+U(b).
 \end{aligned}
\]
Thus $U$ is real linear.
\end{proof}

\section{Proof of the rigidity theorem}\label{sec:proof}
\begin{proof}[Proof of \cref{thm:main}]
Suppose first that $n\ge2$ and use $F$ from \eqref{eq:initial-map}.
By \cref{cor:linearity}, for every unit $v=\sum_jv_je_j$ and $x\in X$,
\begin{equation}\label{eq:dual-coordinate}
 u_v(x)=U(v)(x)=\sum_jv_j u_{e_j}(x)=v\cdot F(x).
\end{equation}
Consequently
\begin{align}
 |F(x)-F(x')|
 &=\sup_{|v|=1}|v\cdot(F(x)-F(x'))|\nonumber\\
 &\overset{\eqref{eq:dual-coordinate}}{=}
        \sup_{|v|=1}|u_v(x)-u_v(x')|
 \le d(x,x').\label{eq:vector-nonexpansion}
\end{align}
Equation \eqref{eq:initial-map} gives $F\circ\iota=\id_{\partial C}$.
Using \eqref{eq:main-input} and composition of pushforwards,
\begin{equation}\label{eq:map-boundary}
 \begin{aligned}
 F_\#\partial T&=F_\#\iota_\#\cur{\partial C}\\
              &=(F\circ\iota)_\#\cur{\partial C}\\
              &=\cur{\partial C}.
 \end{aligned}
\end{equation}
Since $\iota$ is an isometric embedding,
it preserves ordinary AK mass. Thus \eqref{eq:main-input} and
\eqref{eq:convex-current} give
\begin{equation}\label{eq:boundary-mass}
 \begin{aligned}
 \M(\partial T)&=\M(\iota_\#\cur{\partial C})\\
              &=\M(\cur{\partial C})\\
              &=\Haus^{n-1}(\partial C).
 \end{aligned}
\end{equation}

Set $Y=\set T$ and $f=F|_Y$. The mass concentration on $\set T$
and $X=\spt\mu$ give
\[
 \mu(X\setminus Y)=0,\qquad
 \mu(B_r(x)\cap Y)=\mu(B_r(x))>0
 \quad(x\in X,\ r>0).
\]
Thus $\overline Y=X$. Since $X$ is compact and hence complete,
$\widehat Y=X$, and $(Y,T)$ is an integral current space.
By \eqref{eq:vector-nonexpansion}, $\Lip(f)\le1$.
Its extension to $\widehat Y=X$ is $F$, so
$f_\#\partial T=F_\#\partial T=\cur{\partial C}$ by
\eqref{eq:map-boundary}.
Equations \eqref{eq:main-input} and \eqref{eq:boundary-mass} give
$\M(T)\le\Leb^n(C)$ and
$\M(\partial T)\le\Haus^{n-1}(\partial C)$, both with equality.
By \cref{thm:given-map}, $f$ is an isometry onto $C$.
For $x,x'\in X$, take $y_j,y_j'\in Y$ tending to $x,x'$. Then
\[
 |F(x)-F(x')|=\lim_j|f(y_j)-f(y_j')|
             =\lim_jd(y_j,y_j')=d(x,x').
\]
Since $Y$ is dense in $X$ and $F$ is continuous,
\[
 F(X)\subset\overline{F(Y)}=\overline C=C,\qquad
 C=f(Y)\subset F(X).
\]
Thus $F:X\to C$ is a surjective isometry.

Set $D=F_\#T-\cur C$. This is a normal $n$-current in $\R^n$ with
\[
 \spt D\subset F(X)\cup C\Subset\R^n.
\]
By \eqref{eq:push-boundary} and \eqref{eq:map-boundary},
\[
 \begin{aligned}
 \partial D&=F_\#\partial T-\partial\cur C\\
           &=\cur{\partial C}-\cur{\partial C}=0.
 \end{aligned}
\]
By \cite[Theorem~7.2(2)]{Lang}, $D$ has a Euclidean density
$b\in\mathrm{BV}_{loc}(\R^n)$ satisfying
\[
 D=[b],\qquad |Db|=\|\partial D\|=0.
\]
Since $\R^n$ is connected, $b=c$ almost everywhere for some constant $c$.
The identity $\|D\|=|b|\Leb^n$ and compactness of $\spt D$ force
$c=0$, so $D=0$. Hence
\begin{equation}\label{eq:current-recovery}
 F_\#T=\cur C.
\end{equation}
For the inverse isometry $I=F^{-1}:C\to X$,
\[
 I_\#\cur C=I_\#F_\#T=(I\circ F)_\#T=T.
\]
Also $F(\iota(z))=z$ for $z\in\partial C$, so $I|_{\partial C}=\iota$.

If $n=1$, write $C=[\alpha,\beta]$ with $\alpha<\beta$, and let
$a=\iota(\alpha)$, $b=\iota(\beta)$. Then
\[
 \partial T=\cur b-\cur a,\qquad
 \M(T)=\beta-\alpha=d(a,b).
\]
By \cref{lem:two-endpoints}, $T=\cur\gamma$ for an isometric
curve $\gamma:[0,\beta-\alpha]\to X$ from $a$ to $b$.
Since $\gamma$ is isometric,
$\|T\|=\gamma_\#(\Leb^1\llcorner[0,\beta-\alpha])$, so
$X=\spt T=\gamma([0,\beta-\alpha])$.
Define $I(t)=\gamma(t-\alpha)$. It is a surjective isometry extending
$\iota$, and \eqref{eq:curve-current} gives $I_\#\cur C=T$.
\end{proof}

\textbf{Acknowledgement:} The AI tool ChatGPT (OpenAI) was used in the preparation of this manuscript. The author has reviewed all AI-assisted content and takes full responsibility for the final manuscript.

\printbibliography[title={References}]
\end{document}